\documentclass{amsart}
\usepackage{graphicx} 
\usepackage{amsmath,amsfonts,bm}
\usepackage{latexsym,amssymb,verbatim}
\usepackage{paralist}
\usepackage{ulem}
\usepackage{cancel}
\usepackage{mathrsfs} 
\usepackage{graphics} 
\usepackage{epsfig} 
\usepackage{epstopdf}
\usepackage{verbatim}
\usepackage{overpic}
\usepackage{tikz}
\usepackage[colorlinks=true]{hyperref}
\hypersetup{urlcolor=blue, citecolor=red}

\definecolor{verdeoscuro}{rgb}{0.0, 0.5, 0.0}

\newtheorem{lema}{Lemma}[section]
\newtheorem{theo}[lema]{Theorem}
\newtheorem{prop}[lema]{Proposition}

\theoremstyle{definition}

\theoremstyle{remark}
\newtheorem{rema}[lema]{Remark}

\renewcommand{\div}{\operatorname{div}}

\newcommand{\R}{\mathbb{R}}

\newcommand{\ep}{\varepsilon}

\newcommand{\msp}{\mspace{1mu}}

\newcommand{\TCL}[1]{\mathcal{C}^\mathcal{T}_{#1}(\mathbf{0})}

\title[Tangential center problem]{Tangential center problem for  planar Hamiltonian differential equations}
\author{M.J. \'Alvarez}
\address{Departament de Matem{\`a}tiques i Inform{\`a}tica, IAC3 Institute of Applied Computing \& Community Code, Universitat de les Illes Ba\-lears, 07122, Palma de Mallorca, Spain}
\email{chus.alvarez@uib.es}

\author{J.L. Bravo}
\address{Departamento de Matematicas, Universidad de Extremadura, 06071 Badajoz, Spain}
\email{trinidad@unex.es}

\author{L.A. Calder\'on*}
\address{Departamento de Matemáticas, Universidad de Castilla-La Mancha, 13071 Ciudad Real, Spain}
\email{luisangel.calderon@uclm.es}

\author{S. Rebollo-Perdomo}
\address{Departamento de Matemática, Universidad del Bío-Bío, Concepción, Chile}
\email{srebollo@ubiobio.cl}



\begin{document}
	
	\begin{abstract}
		Consider a planar Hamiltonian differential equation with a nondegenerate center
		at the origin. The Tangential Center Problem asks for a characterization of polynomial perturbations for which the center persists at first order. This problem plays a fundamental role in the study of the center-focus problem. We propose an effective algorithmic approach to the Tangential Center Problem for planar polynomial Hamiltonian differential equations. Using this approach, we obtain a complete solution of the quadratic Tangential Center Problem. We further apply it to a non-generic family of cubic Hamiltonian differential equations containing the planar Hamiltonian Lotka--Volterra equation. For a general cubic polynomial perturbation, we derive necessary and sufficient conditions for the vanishing of the corresponding first Poincar\'e--Pontryagin--Melnikov function and characterize all tangential centers. 
	\end{abstract}
	
	\maketitle
	
	\section{Introduction}
	
	This work addresses one of the central problems in the qualitative theory of planar dynamical systems, namely the center--focus problem: given a planar analytic differential equation with an isolated singularity at the origin, determine whether all nearby orbits are closed (in which case the origin is a center) or spiral towards or away from the origin (in which case the origin is a focus).

	The problem dates back to Poincaré and was formally solved by Lyapunov (see \cite{IlyaYako} or ~\cite{Lyapunov} for further details). His solution relies on the computation of the
	so-called Lyapunov constants. However, despite the theoretical completeness of this approach, explicitly determining these constants for a given differential equation is extremely difficult in practice. Consequently, the center--focus problem remains largely unresolved from a practical standpoint, even for low-degree families such as
	planar polynomial differential equations of degree three
	\cite{BIRS,Gasull}.
	
	In light of this situation,
	several approaches and methods have been developed
	to tackle the problem. Among them are: the construction of Lyapunov functions following the foundational work of Lyapunov \cite{Chav,PLC,Schlomiuk}, normal form theory  \cite{Andreev17,HaWa,LiuLi}, analysis of the power series expansion
	of solutions in polar coordinates \cite{ARJMAA03,ALGM,GGM}, algebraic approaches involving the structure of the Lyapunov constants \cite{CGMM,Joyal}, the Lyapunov-Schmidt reduction method \cite{FLLL}, and averaging techniques \cite{SV}. 
	
	Given the complexity of the center-focus problem, several authors have proposed graded classifications of centers into different levels, based on analytic, algebraic and geometric criteria. A particularly promising approach is 
	the study of the persistence of centers in polynomial Hamiltonian equations under small polynomial perturbations. This framework aims to make the problem more tractable by identifying intermediate structures that govern the unfolding and stability of centers.

	The general question we are address is the following: given a Hamiltonian differential equation with a center at the origin, characterize all polynomial perturbations for which the origin remains a center. This problem is known in the literature as the Infinitesimal Center Problem. In this article, we focus on its first-order version, which will be defined more precisely below and is known as the Tangential Center Problem.

	Both the Tangential Center Problem and the Infinitesimal Center Problem have been widely studied. 
	For instance, the Tangential Center Problem has been solved for Hamiltonians $H$ that factor into products of lines in generic position \cite{Uribe,PelletierUribe}, as well as for products of polynomials satisfying certain general-position assumptions~\cite{PontigoHerrera}. It has also been solved for hyperelliptic Hamiltonians~\cite{CM}. Variants of both problems have been formulated and solved in other settings, including Abel equations with polynomial coefficients~\cite{PM,Pakovich3} and zero-dimensional cycles~\cite{ABCM}.
	Despite substantial progress, both problems remain open in general.
	
	In this work, we focus exclusively on the Tangential Center Problem.
	Following a computational approach, we propose an algorithm for analyzing this problem, which enables us to solve it for low-degree differential equations. More precisely, we completely solve the quadratic Tangential Center Problem, which, to the best of our knowledge, has remained open in the literature. Our result therefore extends previous results that were obtained only under genericity assumptions.

	Finally, to illustrate the algorithm in a more challenging setting, we  consider the Tangential Center Problem for cubic planar Hamiltonian differential equations associated to the Hamiltonian function
	\begin{equation*}
		H(x,y)=\frac{x^2+y^2}{2}+x^2y-\frac{y^3}{3}+ax^4+bx^2y^2+cy^4.
	\end{equation*}
	This family generalizes the quadratic Hamiltonian Lotka--Volterra system by adding a homogeneous quartic term while preserving one of its symmetries.  The results of Ilyashenko~\cite{Ilyashenko} solve both the Tangential and Infinitesimal Center Problems for generic Morse Hamiltonians, showing that only exact perturbations solve either problem. This motivates us to consider non-generic families. In particular, we study a non-generic cubic Hamiltonian equation that contains the most degenerate quadratic Hamiltonian equation. We then consider a generic perturbation by a cubic (1)-form and solve the corresponding Tangential Center Problem.
	
	The paper is organized as follows. In Section~\ref{sec:prel}, we introduce the necessary notation and formulate the Tangential Center Problem. The main results of this work are established in Section~\ref{main-results}. The methodology and the proposed algorithm to tackle the Tangential Center Problem are given in Section~\ref{sec:3}. In Section~\ref{sec:proofs-quad}, we analyse the Tangential Center Problem for quadratic Hamiltonian differential equations. Finally,  in  Section~\ref{sec:proofs-cubic}, we analyse the Tangential Center Problem for a family of cubic  Hamiltonian differential equations. For the quadratic and the cubic families we also study the relation of the solutions of the Tangential Center Problem with the intersections of the families of Hamiltonian centers with other families of centers. 
	
	\section{Notation and problem formulation}\label{sec:prel}
	
	In order to state the main results of this work, we first introduce the necessary notation and formulate the main problems addressed in this paper. 
	
	Fix a polynomial $H \in \mathbb{R}_{n+1}[x,y]$ and suppose that it has a nondegenerate minimum at the 
	origin $\mathbf{0}:=(0,0)\in\R^2$, and that $H(\mathbf{0}) = 0$. Then, the associated {\it Hamiltonian differential equation}
	\[
	dH =H_x\, dx+H_y\, dy= 0,
	\] 
	where $H_x$ and $H_y$ denote the partial derivatives of $H$, has a nondegenerate center at the origin,  surrounded by 
	a continuous family ${ \{ \gamma_h \mid h \in (0, \hat h) \} }$
	of {\it cycles} (periodic orbits) of $dH=0$, parametrized by the energy level $h$ of $H$ and defined on a maximal open interval $(0,\hat h)\subset \R^{+}$.  This family of cycles foliates the {\it period annulus $\mathcal{A}$ of the origin}, that is,
	$\mathcal{A}=
	\cup_{h \in (0, \hat h)}\gamma_h.$

	Denote by $\varOmega^1_n[x,y]$ the {\it space of polynomial $1$-forms of degree at most $n$}. Given
	$\omega \in \varOmega^1_n[x,y]$ 
	we can consider the following two constructions. On the one hand, for a small real parameter 
	$\varepsilon \in \R$,  the {\it perturbed Hamiltonian differential equation} is
	\begin{equation}
		\label{eq:perturbed-r-gen}
		dH + \varepsilon \msp \omega = 0.
	\end{equation}
	
	On the other hand, the {\it Abelian integral defined by $H$ and $\omega$} is
	\begin{equation}
		\label{eq:Abelian-integ-gen}
		I\colon (0,\hat h) \longrightarrow \R,
		\quad
		h \longmapsto \oint_{\gamma_h} \omega,
	\end{equation}
	which is an analytical real-valued function. These two objects are related to our aim: to study the persistence of the center of the Hamiltonian differential equation $dH=0$ under the perturbation $\varepsilon \omega$.

	Henceforth, we consider only $1-$forms $\omega\in \varOmega^1_n[x,y]$ for which the origin remains a singular point of equation \eqref{eq:perturbed-r-gen}. We denote by $\varOmega^1_n(\mathbf{0})$ the subset of $1$-forms $\omega\in \varOmega^1_n[x,y]$ that satisfy this condition.
	
	A usual tool to study how the perturbation $\varepsilon \omega$ modifies the foliation defined by $dH=0$, close to the origin, is the {\it displacement function associated with equation \eqref{eq:perturbed-r-gen}}:
	$$
	L\colon (-\ep_0,\ep_0) \times (0,h^*)
	\longrightarrow \R,
	\quad
	(\varepsilon, h) \longmapsto  \mathcal{P}(\varepsilon, h)-h,
	$$
	where $\ep_0>0$, $0 < h^* \leq \hat{h}$ and $\mathcal{P}(\varepsilon, h)$ is the first return map. Since $\mathcal{P}(\varepsilon, h)$ is an analytic function, the displacement function can be written as the power series
	$$
	L(\varepsilon, h) = \varepsilon M_1(h) + \varepsilon^2 M_2(h) + \cdots + \varepsilon^k M_k(h) + O(\varepsilon^{k+1}).
	$$
	The coefficient $M_i(h)$ is  the $i$-th {\it Poincar\'{e}--Pontryagin--Melnikov (PPM) function}. 
	
	Since we assume that $H$ is fixed and that we know which 1-form $\omega$ will be used for the perturbation, we have not made explicit the dependence on $H$ and $\omega$ in $I(h)$ and $L(\varepsilon,h)$.
	
	By construction, each zero of 
	$L(\varepsilon,h)$ in the variable $h$ is in correspondence with a cycle of 
	$dH+\varepsilon \msp \omega=0$. Hence, if the displacement function $L(\varepsilon, h)$ vanishes identically for small enough values of $h$, then
	the perturbed Hamiltonian differential equation $dH+\varepsilon \msp \omega=0$ has a center at the origin for each $\varepsilon$ small enough. That means that
	the center of $dH=0$ persists under the perturbation.

	Given a Hamiltonian differential equation ${dH=0}$ of degree $n$ with a nondegenerate center at the origin, the Infinitesimal Center Problem  consists in determininig the set of 1-forms $\omega \in \varOmega^1_n(\mathbf{0})$ for which there exist  $h^*>0$ and $\varepsilon^*>0$  such  that $L(\varepsilon,h)\equiv 0$ 
	for all $ h\in(0,h^*),\ 
	\varepsilon\in(-\varepsilon^*,\varepsilon^*).$
	
	According to the expansion in power series of $L(\varepsilon,h),$  the persistence of a center generally requires
	$$
	M_1(h)\equiv M_2(h)\equiv \cdots \equiv M_k(k)\equiv \cdots\equiv 0.
	$$
	The classical Poincaré--Pontryagin formula says that the first-order PPM function $M_1(h)$ is given by the Abelian integral \eqref{eq:Abelian-integ-gen}. Hence, the first step towards the solution of the Infinitesimal Center Problem is to solve  the following problem.
	
	\medskip
	\noindent{\bf Tangential Center Problem.}
	Given a Hamiltonian differential equation $dH=0,$ of degree $n$, with a nondegenerate center at the origin,
	to determine the set
	$$
	\TCL{n}:=
	\left\{ \omega \in \varOmega^1_n(\mathbf{0}) \, \mid \, \text{ exists } h^*>0 \text{ such that } I(h)= 0 \text{  for all } h \in (0, h^*)\right\}.
	$$
	The adjective {\it tangential} reflects  the first-order (linearized) condition for the persistence of the center under the perturbation. 
	
	Observe that
	the inclusion $
	\TCL{n} \subset \varOmega^1_n(\mathbf{0})
	$ is proper since, for example, for the 1-form $x\, dy$ the corresponding Abelian integral $I(h)$ represents the area enclosed by the cycle $\gamma_h$ which clearly is nonzero for any $0<h<\hat{h}$.

	To simplify the statement of the main results, we introduce the following notation.
	We denote the {\it subspace of exact polynomial 1-forms} by
	\[
	\varOmega^1_{\mathrm{e}}
	=\{\omega\in \varOmega^1_n(\mathbf{0}) \;|\; \omega=dF,\, F \text{ polynomial}\},
	\]
	and the {\it subspace of algebraically relatively exact polynomial $1$-forms} by
	\[
	\varOmega^1_{\mathrm{re}}
	=\{\omega\in \varOmega^1_n(\mathbf{0}) \;|\; \omega=dF+G\,dH,\, F,G \text{ polynomials}\}.
	\]
	
	If for some $h\in\mathbb{R}$, $H-h=\prod_{j=1}^m f_j$, with each $f_j$ irreducible and $f_j(0,0)\neq 0$, then {\it the space of  logarithmic $1$-forms associated with $f_j$}, is denoted by
	\[
	\varOmega^1_{\log}(f_j)
	=\left\{\omega\in \varOmega^1_n(\mathbf{0}) \;|\; \omega = \lambda_j (H-h) \frac{df_j}{f_j},\; \lambda_j\in\mathbb{R}\right\},\quad j=1,\ldots,m.
	\]
	
	Finally, if $H$ admits a global orientation-preserving $\mathbb{Z}_k$-symmetry fixing the origin, that is, an orientation-preserving diffeomorphism
	$\sigma\colon\mathbb{R}^2\to\mathbb{R}^2$  such that
	\[
	\sigma^k:=\underbrace{\sigma \circ \cdots \circ \sigma}_{\text{$k$-times}} =\mathrm{Id} 
	, \:
	\sigma^i\neq \mathrm{Id}\, \mbox{ for $1\leq i< k$},\:
	\sigma(\mathbf{0})=\mathbf{0}
	\quad
	\mbox{and}
	\quad
	H\circ\sigma = H,
	\]
	we define the {\it space of $\sigma$-balanced 
		polynomial $1$-forms} by
	\[
	\varOmega^1_{\sigma}
	=\left\{\omega\in \varOmega^1_n(\mathbf{0}) \colon\sum_{i=0}^{k-1}(\sigma^i)^*\omega=0\right\}.
	\]
	If instead $\sigma$ reverses the orientation, then 
	\[
	\varOmega^1_{\sigma}
	=\left\{\omega\in \varOmega^1_n(\mathbf{0}) \colon\sum_{i=0}^{k-1}(-1)^i(\sigma^i)^*\omega=0\right\}.
	\]

	\section{Main results}
	\label{main-results}
	
	Throughout this work, a polynomial is said to be in {\it canonical form} if its quadratic part is given by ${(x^2+y^2)/2}$,
	a form that every real polynomial with a nondegenerate minimum at the origin attains after a suitable linear change of coordinates.

	We  first consider the  Tangential Center Problem for quadratic Hamiltonian equations. To this end, we  use 
	the canonical forms 
	of cubic polynomials arising from the quadratic Hamiltonian equations given by Iliev \cite[p. 108]{Iliev}, who makes a refinement of the classification of quadratic centers given by Zoladek \cite{Zoladek}. It is usual to distinguish between the cubic generic case and a cubic degenerate situation commonly referred to as the Hamiltonian triangle (which is the intersection of the Hamiltonian and the Lotka--Volterra families).
	
	The cubic generic case is described by the polynomial
	\begin{equation}\label{Hamilt-q}
		H_N(x,y)=\frac{x^{2}+y^{2}}{2}
		+(x^{2}+y^{2})y
		+\frac{3x^{2}-y^{2}}{3}\,a y
		-\frac{x^{2}-3y^{2}}{3}\,b x,
	\end{equation}
	where $a,b\in\mathbb{R}$ are real parameters (we have renamed the parameters used by Iliev), while
	the Hamiltonian triangle corresponds to the polynomial
	\begin{equation}\label{eq:HamT}
		H_T(x,y)=\frac{x^2+y^2}{2}+x^2y-\frac{y^3}{3},
	\end{equation}
	
	Our main result concerning quadratic Hamiltonian equations is the following, 
	which will be proved in Section~\ref{sec:proofs-quad}.

	\begin{theo}\label{theo:mainqintro}
		Consider the quadratic Hamiltonian equation $dH_N=0$, associated with the generic cubic polynomial $H_N(x,y)$ given in \eqref{Hamilt-q}, and the quadratic Hamiltonian equation $dH_T=0$, associated with the polynomial
		$H_T$ given in \eqref{eq:HamT}. 
		
		\begin{enumerate}
			\item If  $b\neq 0$, then the Tangential Center Problem for $dH_N=0$ has the solution:
			\[
			\TCL{2} = \varOmega^1_{e}.
			\] 
			
			\item If  $b=0$, then $H_N(x,y)$ has a orientation-reversing $\mathbb{Z}_2$-symmetry $\sigma$, which is a reflection on the $y$-axis; and the Tangential Center Problem for  $dH_N=0$ has the solution:
			\[
			\TCL{2} = \varOmega^1_{e}  \oplus \varOmega^1_{\sigma}.
			\]
			\item The polynomial $H_T(x,y)$ has an orientation-reversing $\mathbb{Z}_2$-symmetry $\sigma$, which is a reflection on the $x$-axis, and a orientation-preserving $\mathbb{Z}_3$-symmetry $\tau$, which is a  $2\pi/3$-rotation; and the Tangential Center Problem for  $dH_T=0$ has the solution:
			\[
			\TCL{2} = \varOmega^1_{e}  \oplus \varOmega^1_{\sigma} \oplus \varOmega^1_{\tau}.
			\]
		\end{enumerate}
	\end{theo}
	
	After the quadratic equation case, we  study cubic Hamiltonian ones. Concretely, in Section~\ref{sec:proofs-cubic}, we study the family of Hamiltonian equations with a triangle quadratic part, and a cubic part  with odd exponents, namely, the family given by $dH=0$ with 
	\begin{equation}\label{eq:Hamce0}
		H(x,y)=\frac{x^2+y^2}{2}+x^2y-\frac{y^3}{3}+ax^4+bx^2y^2+cy^4.
	\end{equation}
	Note that $H(-x,y)=H(x,y)$ for any $(x,y)$, that is, the entire family is symmetric with respect to the $y$ axis. Denote by $\sigma$  that symmetry. 
	
	\begin{theo}\label{theo:cubic}
		Consider the cubic Hamiltonian equation $dH=0$, associated with the quartic polynomial $H(x,y)$ given in \eqref{eq:Hamce0}, which always has an orientation-reversing $\mathbb{Z}_2$-symmetry $\sigma$, which is a reflection with respect to the $y$-axis. The solution of the Tangential Center Problem for $dH=0$ is as follows:
		\begin{enumerate}
			\item If $c=a$, $b=2a$, then $H$ has an orientation-preserving $\mathbb{Z}_3$-symmetry $\tau$, which is a $2\pi/3$-rotation, and
			\[ 
			\TCL{3} =\varOmega^1_{e}  \oplus \varOmega^1_{\sigma}\oplus \varOmega^1_{\tau}.
			\]
			
			\item If $c\neq a$, $b=c=0$, then
			\[
			\TCL{3}=\varOmega^1_{re}  \oplus \varOmega^1_{\sigma}.
			\]
			
			\item If $c\neq a$ and 
			\[a+b+c=18 \, b^{2} c + 72 \, b c^{2} + 72 \, c^{3} - b^{2} - 2 \, b c + 8 \, c^{2}=0,\]
			then 
			$H(x,y) - k(a,b,c) = f_1 f_2 f_3,$ with $k$ a certain constant depending on $a,b,c$ and $f_1,f_2$ linear and $f_3$ quadratic, and
			\[
			\TCL{3}=\varOmega^1_{e}  \oplus \varOmega^1_{\sigma} \oplus \varOmega^1_{\log}(f_1)\oplus \varOmega^1_{\log}(f_2).
			\]
			\item In any other case, 
			\[
			\TCL{3} =\varOmega^1_{e}  \oplus \varOmega^1_{\sigma}.
			\]
		\end{enumerate}
	\end{theo}
	\begin{rema}
		Observe that the normal forms \eqref{Hamilt-q} and \eqref{eq:HamT} show that, for quadratic Hamiltonian differential equations, any center at the origin is necessarily nondegenerate. Hence, the nondegeneracy assumption, which has been maintained throughout this paper, is not restrictive in the quadratic case. Moreover, the family of cubic Hamiltonian equations defined by \eqref{eq:Hamce0} that we have chosen also satisfies this assumption.
	\end{rema}

	\begin{rema} 
		Observe that case (2) of the previous result is the only one in which the relatively exact 1-forms arise. Recall that the perturbation has the same degree as the unperturbed equation. If perturbations of higher degree were allowed, relatively exact 1-forms would arise more frequently. Nevertheless, despite this degree restriction, relatively exact 1-forms do appear in this case, as the symmetry causes the higher-degree terms to cancel out.
	\end{rema}

	\section{Methodology}\label{sec:3} 
	
	In this section, we present the computational methodology underlying our approach, together with the proposed algorithm to tackle the Tangential Center Problem.
		
	Let $H(x,y)$ be a polynomial with a nondegenerate minimum at the origin $\mathbf{0}$ such that $H(\mathbf{0}) = 0$. Let $\omega=P(x,y)\,dx+Q(x,y)\,dy \in \Omega_n^1(\mathbf{0})$ and consider the Abelian integral \eqref{eq:Abelian-integ-gen}, defined by $H$ and $\omega$, that is, 
	\[
	I(h)=\oint_{\gamma_h}\omega.
	\]
	Recall that we are interested in determining the conditions on $\omega$ under which there exists $h^*>0$ such that
	$I(h)= 0$ for all $h\in(0,h^*)$.
	
	To simplify the computations of the Abelian integral, we reparameterize the energy levels by writing 
	\[
	H(x,y)=\frac{h^2}{2},
	\]
	instead of $H(x,y)=h$. Since this change merely reparameterizes the family of cycles of the period annulus $\mathcal{A}$ of the origin, we retain the notation $\gamma_h$ for the cycles surrounding it. 
	
	Observe that $I(h)\equiv0$ whenever $\omega$ is exact. Therefore, it suffices to study \(\omega\) modulo exact differential forms. In particular, it is always possible to find a representative of the equivalence class \([\omega]\) of the form $F(x,y)\,dx$, since \(\omega = F(x,y)\,dx + d\varphi\), where 
	\[
	\varphi(x,y) := \int_{0}^{y} Q(x,t)\,dt,\quad F(x,y) := P(x,y) - \int_{0}^{y} \frac{\partial Q}{\partial x}(x,t)\,dt.
	\]
	That is, the quotient space of polynomial $1$--forms modulo exact forms admits a canonical basis given by
	\begin{equation}\label{basispetrov}
		\mathcal{B}_n = \left\{ x^i y^j\,dx \;\middle|\; 0 \leq i \leq n,\; 1 \leq j \leq n,\; 1 \leq i+j \leq n \right\}.
	\end{equation}
	Thus, every equivalence class modulo exact forms admits a unique representative in the span of the elements of \eqref{basispetrov}. Consequently, it is enough to consider perturbations given by arbitrary linear combinations of these basic polynomial $1$--forms. Restricting our attention to this canonical set of representatives, we can now describe the methodology that will be used to establish our main results.

	\begin{prop}\label{prop-base-of-algo}
		Let $H(x,y;\bm{\beta})$ be a family of polynomials of degree $n+1$, parameterized by a finite (possibly empty) set of parameters $\bm{\beta}=(\beta_1,\ldots, \beta_s)$, each with a nondegenerate minimum at the origin and in canonical form.
		Consider $\omega \in \Omega_n^1(\mathbf{0})$. Then, the Abelian integral defined by $H(x,y;\bm{\beta})$ and $\omega$
		admits a convergent power series
		expansion
		\[
		I(h)=\sum_{\mu=2}^{\infty} K_\mu\,h^\mu,
		\]
		where each $K_\mu$ depends linearly
		in the coefficients of $\omega$ and 
		polynomially on $\bm{\beta}$.
		Moreover,
		there is $N\in\mathbb{N}$ depending only on $n$ such that $K_2=K_3=\cdots=K_N=0$ implies $I(h)\equiv0.$
	\end{prop}

	\begin{proof}
		Since the  origin is a nondegenerate center of
		$dH=0$,  we may introduce 
		polar coordinates $(r,\theta)$ 
		in which $H(x,y;\bm{\beta})$ writes as
		\[
		H(r,\theta;\bm{\beta})=\frac{r^2}{2}+C_3r^3+\cdots+C_{n+1}r^{n+1},
		\]
		where each $C_i$ is a trigonometric polynomial in $\theta$ whose coefficients are polynomials in $\bm{\beta}$.

		For $h>0$ sufficiently small and $\theta$ fixed, the equation 
		$H(r,\theta)=h^2/2$ defines a real algebraic curve in the $rh$-plane with an ordinary double point at the origin. Hence,  in a neighbourhood of the origin, the algebraic curve has two analytic branches, one of them belonging to the first quadrant, which implies that $r$ can be (locally) analytically expressed  in terms of $h$ as a power series\footnote[2]{The branches of the algebraic curve $H(r,\theta;\bm{\beta})=h$ can  be parameterized via Puiseux series involving  powers of $h^{\frac{1}{2}}$. To avoid these fractional power we have chosen the reparameterization of the energy levels of $H.$} of the form
		\begin{equation}
			\label{eq:r-series-proof}
			r(h,\theta)
			=
			h+\sum_{\mu=2}^{\infty} r_{\mu}(\theta)h^{\mu}.
		\end{equation}
		where the coefficients $r_{\mu}(\theta)$, $\mu=2,\ldots$, can be obtained recursively by replacing $r$ in the equation $H(r,\theta)=h^2/2$ by the expression \eqref{eq:r-series-proof}, obtaining in particular that each coefficient $r_{\mu}(\theta)$ is a trigonometric polynomial whose
		coefficients depend polynomially on the parameters $\bm{\beta}$ of $H$. Therefore, since $\theta$ can vary from $0$ to $2\pi$, the cycle $\gamma_h$ admits
		a real-analytic parametrization of the form
		\begin{equation*}
			\label{eq:gamma-param}
			\varphi(\theta)=
			\big(r(h,\theta)\cos\theta, r(h,\theta)\sin\theta \big).
		\end{equation*}

		Now, let $\omega\in\Omega^1_n(\mathbf{0})$. As we commented before, it is enough to consider monomial $1$--forms  like the ones in \eqref{basispetrov}. That is, we can always write
		\[
		\omega=\omega_{ex}+\sum_{i+j=1}^{n}\alpha_{ij}\, \omega_{ij},
		\quad 0 \leq i \leq n,\; 1 \leq j \leq n,
		\]
		where $\omega_{ex}\in\Omega^1_{e}$, $\alpha_{ij}\in\mathbb{R}$, and $\omega_{ij}=x^iy^jdx$. Using the parametrization $\varphi(\theta)$
		we obtain
		\[
		\oint_{\gamma_h}\omega_{ij}
		=
		\int_0^{2\pi}
		r^{i+j}(h,\theta)
		\cos^i\theta\,\sin^j\theta\,
		\bigl(r_\theta(h,\theta)\cos\theta-r(h,\theta)\sin\theta\bigr)
		\,d\theta.
		\]
		
		Substituting the series expansion \eqref{eq:r-series-proof} for $r(h,\theta)$ and
		$r_\theta(h,\theta)$ into this expression and expanding in powers of $h$, we
		obtain a convergent power series
		\[
		\oint_{\gamma_h}\omega_{ij}
		=
		\sum_{\mu=1}^{\infty} K_{ij,\mu}\,h^{\mu},
		\]
		where each coefficient $K_{ij,\mu}$ is obtained by integrating a trigonometric
		polynomial in $\theta$. Consequently, each $K_{ij,\mu}$ is a polynomial function on
		the parameters $\beta$ of $H$, if they exist, and depends linearly on the coefficients of the $1$--form
		$\omega$.
		
		By linearity, the same conclusion holds for the Abelian integral
		\[
		I(h)=\oint_{\gamma_h}\omega=\oint_{\gamma_h}\left(\omega_{ex}+\sum_{i+j=1}^n\alpha_{ij}\omega_{ij}\right)=\sum_{i+j=1}^n\alpha_{ij}\oint_{\gamma_h}\omega_{ij},
		\]
		which therefore admits a convergent expansion
		\[
		I(h)=\sum_{\mu=2}^{\infty} K_{\mu}\,h^{\mu},
		\]
		where each coefficient $K_{\mu}$ depends linearly
		on the coefficients $\bm{\alpha}=\{\alpha_{ij}\}$ of $\omega$ and 
		polynomially on $\bm{\beta}$.
		
		Finally, consider the ideal $\mathcal{I}$ generated by the coefficients
		$\{K_{\mu}\}_{\mu\ge1}$ in the polynomial ring generated by the coefficients of $H$ and
		$\omega$. By the Hilbert Basis Theorem, this ideal is finitely generated. Hence,
		there exists an integer $N\ge 2$ such that if $K_2=K_3=\cdots=K_N=0$, then $K_{\mu}=0$ for all $\mu\ge 2$, and obviously $I(h)\equiv 0$. 
	\end{proof}

	\subsection{Algorithmic computation of the necessary conditions}
	
	We now provide an overview of the algorithm we are going to follow in order to characterize the $1$--forms for which $I(h)\equiv 0.$ The algorithm is based on the previous result, Proposition \ref{prop-base-of-algo}. 
	
	\noindent {\bf Input}: A family of polynomial functions \( H(x,y;\bm{\beta}) \) of degree \( n+1 \), depending on parameters \(\bm{\beta}=\{ \beta_1, \dots, \beta_s \} \) and \( \omega \in \Omega^1_n(\mathbf{0}) \) with coefficients $\bm{ \alpha}=\{\alpha_{ij}\}$. 
	
	\noindent {\bf Output:} A finite number of algebraic varieties in the parameters $\bm{\alpha}$, $\bm{\beta}$,  linear in $\bm{\alpha}.$
	
	\begin{enumerate}
		\item[Step $1$.] \textbf{Polar parametrization of $\gamma_h$}: Write $H(x,y;\bm{\beta})$  in polar coordinates to obtain 
		a real-analytic polar parametrization of $\gamma_h$ in the form
		\eqref{eq:r-series-proof}.
		
		\item[Step 2] \textbf{Computation of the power series of $\bm{I(h)}$:}  As previously discussed, we can write
		\[
		\omega = \omega_{ex} + \alpha_{ij} \sum_{i,j \in A} \omega_{ij},
		\]
		where \( \alpha_{ij} \in \mathbb{R} \), \( \omega_{ij} \in \mathcal{B}_n \), and \( \omega_{ex} \in \Omega^1_{e}\). Next, following the process described in the proof of the previous result, we can express
		\[
		I(h) = \sum_{\mu=1}^\infty K_{\mu} h^{\mu},
		\]
		where the \( K_{\mu} \) are computable expressions that depend linearly on \(\bm{\alpha}\) and polynomially on \( \bm{\beta} \).
		
		\item[Step 3] \textbf{Computation of the conditions for $\bm{I(h)\equiv 0}$:}  
		According to Proposition~\ref{prop-base-of-algo}, there exists a natural number \(N\) such that, if \(K_{\mu}=0\) for \(2\leq \mu\leq N\), then \(K_{\mu}=0\) for all \(\mu\geq 1\). Hence, a natural strategy for solving our problem is to compute \(K_{\mu}\) for \(2\leq \mu\leq m\), for some \(m\geq 2\), and determine the conditions on \(\bm{\alpha}=\{\alpha_{ij}\}\) and \(\bm{\beta}=\{\beta_l\}\) under which \(K_{\mu}=0\) for all \(2\leq \mu\leq m\). We proceed as follows.
		
		Suppose that \(\mu=2\), or that \(K_{\mu}\equiv 0\) for all \(\mu<m\). The equation \(K_m=0\) is linear  in \(\bm{\alpha}\) and polynomial in \(\bm{\beta}\). Hence, for a fixed coefficient \(\alpha_{i_0j_0}\in\bm{\alpha}\), we can write
		\[
		K_m=p\alpha_{i_0j_0}+q=0,
		\]
		where \(p\) and \(q\) are polynomials in \(\bm{\beta}\). Moreover, \(p\) is a nonzero polynomial that does not depend on any parameter in \(\bm{\alpha}\), whereas \(q\) depends linearly on the remaining parameters in \(\bm{\alpha}\).
		
		Note that, if \(\alpha_{ij}=0\) for all \(\alpha_{ij}\in\bm{\alpha}\), then the only polynomial \(1\)-forms for which \(I(h)\equiv 0\) are the exact ones. Therefore, we divide the problem into two complementary subproblems:
		\begin{enumerate}
			\item[] \(3.m.a\): Solve for \(\alpha_{i_0j_0}\), setting
			$\alpha_{i_0j_0}=-\frac{q}{p},
			$ under the assumption that \(p\neq 0\).
			
			\item[] \(3.m.b\): Impose the additional conditions on \( \alpha_{ij} \) and \( \beta_l \) given by \(p=0\) and \(q=0\).
		\end{enumerate}
		
		In both subproblems, we have \(K_m\equiv 0\).
		
		We then choose one of these subproblems and proceed analogously to obtain \(K_{m+1}\equiv 0\).
		
		At each step, either the number of coefficients in \(\bm{\alpha}\) is reduced by one, or the codimension of the parameter space in \(\bm{\beta}\) decreases by two. This process generates a binary tree of depth at most \(N\). Consequently, the original problem can be reduced to at most \(2^N\) simpler subproblems.
	\end{enumerate}
	
	\begin{rema}
		Another valid approach is to pick an $m\geq 2$, and compute all $K_{\mu}$ for $2\leq \mu \leq m$, construct the algebraic system $K_2=0,\ldots,K_m=0$ and compute the conditions imposed over $\{\alpha_{ij}\}$ and $\{\beta_l\}$ by, for example, computing the Gröbner basis. While this approach is theoretically feasible, it can become computationally challenging if the number of unknowns is too large, and so we have proposed the previous, more computationally efficient method.  Nevertheless, we also use this approach when it is convenient.
	\end{rema}
	
	\subsection{Sufficiency of the conditions}

	The number $N\in \mathbb{N}$, guaranteed by Proposition \ref{prop-base-of-algo}, is called the {\it Bautin number} and
	it is not known a priori.  Hence,
	the algorithm described above cannot be used to directly characterize the Tangential Center Problem. Indeed, the successive vanishing conditions obtained from the expansion of the Abelian integral $I(h)$ are only necessary conditions for the existence of a tangential center. More precisely, one computes the coefficients of the corresponding series expansion until the new conditions generated at a given step belong to the ideal generated by the previous ones. This stabilization process produces an algebraic ideal whose associated prime decomposition yields a collection of algebraic varieties in the coefficient space of $\omega$. These varieties constitute natural candidates for irreducible components of the tangential center locus. 
	
	However, since the truncation is performed without the knowledge of $N$, the resulting varieties need not consist entirely of tangential centers. Consequently, the computational part of the problem must be complemented by a geometric and analytic analysis of each candidate component. For every variety arising from the associated prime decomposition, one must determine whether the corresponding $1$--forms actually satisfy the Tangential Center Problem. This requires identifying the structural mechanisms that force the Abelian integral to vanish identically, and therefore guarantee the vanishing of the first-order displacement function. 
	
	Identifying and characterizing these mechanisms (symmetries, exactness or relative exactness, etc) is therefore a crucial step in the determination of the tangential center locus. Motivated by the works of Movasati \cite{Movasati}, Zare \cite{Zare}, and Pontigo-Herrera \cite{PontigoHerrera} on the irreducible components of the tangential center space for some generic families, we prove the following result. It provides sufficient conditions for the Tangential Center Problem and, at the same time, highlights the geometric and algebraic features underlying the solutions arising from the Hamiltonian families considered in this work.

	\begin{prop}\label{prop:suficientes}
		Let $H\colon\mathbb{R}^2\to\mathbb{R}$ be a  polynomial 
		with a
		nondegenerate minimum at the origin and in canonical form. Let
		$\gamma_h\subset\{H=h\}$ be a cycle of $dH=0$.
		Let $\omega$ be a smooth $1$--form defined in a neighborhood of $\gamma_h$, and
		consider the Abelian integral associated to it.
		
		Assume that for some 
		$h_0\in \R$, $H-h_0=\prod_{j=1}^m f_j$,
		where $f_j$ is not zero at any point of $\gamma_h$, $j=1,\ldots,m$, and 
		$H$ admits global $\mathbb{Z}_{k_i}$-symmetries $\sigma_i$, $i\in \{1,\ldots, l\}$. Let 
		\[
		\omega \in \varOmega^1_{\mathrm{re}}
		\;\oplus\;
		\bigoplus_{j=1}^m \varOmega^1_{\log}(f_j)
		\;\oplus\;
		\bigoplus_{i=1}^l \varOmega^1_{\sigma_i}.
		\]
		Then $I(h)\equiv0$.
	\end{prop}
	
	\begin{proof}
		By lineality, it suffices to prove $I(h)\equiv 0$ for $1$-forms in each of the summands. 
		
		If $\omega\in \varOmega^1_{re}$, that is, $\omega=dF+G\,dH$ for $G,H\in\mathbb{R}[x,y]$, then along the cycle $\gamma_h$ we have
		$dH=0$. Therefore,
		\[
		\oint_{\gamma_h}\omega
		=
		\oint_{\gamma_h} dF
		=
		0,
		\]
		
		If $\omega\in \varOmega^1_{\log}(f_j)$, then,
		$\omega = \lambda_j (H-h_0) \frac{df_j}{f_j},\; \lambda_j\in\mathbb{R}$. Therefore, 
		\[
		\omega = \lambda_j (H-h_0) d\log(f_j),
		\]
		and we conclude as above.
		
		Assume now that $\sigma$ is a $\mathbb{Z}_{k}$-symmetry that satisfies the
		stated properties. Since $\sigma$ leaves the level set $\{H=h\}$ invariant,
		$\sigma(\gamma_h)=\gamma_h$. Assume first that $\sigma$ preserves the orientation. Hence,
		\[
		\oint_{\gamma_h}\omega
		=
		\oint_{\sigma^i(\gamma_h)}\omega
		=
		\oint_{\gamma_h}(\sigma^i)^*\omega
		\qquad
		\text{for all } i=0,\dots,k-1.
		\]
		Adding these identities, we obtain
		\[
		k\,\oint_{\gamma_h}\omega
		=
		\oint_{\gamma_h}
		\sum_{i=0}^{k-1}(\sigma^i)^*\omega.
		\]
		By hypothesis, the integrand on the right-hand side vanishes identically, and
		hence
		\[
		\oint_{\gamma_h}\omega = 0.
		\]
		Lastly, if $\sigma$ reverses the orientation, then 
		\[
		\oint_{\gamma_h}\omega
		=
		\oint_{\sigma^i(\gamma_h)}\omega
		=
		\oint_{\gamma_h}(-1)^i(\sigma^i)^*\omega
		\qquad
		\text{for all } i=0,\dots,k-1.
		\]
		Again adding these identities, we obtain
		\[
		k\,\oint_{\gamma_h}\omega
		=
		\oint_{\gamma_h}
		\sum_{i=0}^{k-1}(-1)^i(\sigma^i)^*\omega.
		\]
		By hypothesis, the integrand on the right-hand side vanishes identically, and
		hence
		\[
		\oint_{\gamma_h}\omega = 0.
		\]
	\end{proof}

	\begin{rema}
		The first two cases correspond to the triviality of the Abelian integral in the relative cohomology group associated with the Hamiltonian $H$, while the symmetric case expresses a cancellation phenomenon induced by a finite symmetry of the level curves.
		Both mechanisms occur naturally in the study of polynomial perturbations of
		Hamiltonian differential equations and in the reduction of Abelian integrals via the Petrov
		module.
	\end{rema}
	
	Given one of the specific subfamilies of Hamiltonians $H(x,y)$ that the algorithm identifies, and a specific $1$--form $\omega^*\in\Omega^1_n(\mathbf{0})$ which is a candidate as a solution for the tangential problem, we can use Proposition \ref{prop:suficientes} to prove that $I(h)=\oint_{\gamma_h}\omega^*\equiv0$ (this is, that indeed $\omega^*\in\TCL{n}$) in the following way. First, by Proposition \ref{prop:suficientes}, if we can prove with an straightforward computation that \[\omega^*\in \varOmega^1_{\mathrm{re}}
	\;\oplus\;
	\bigoplus_{j=1}^m \varOmega^1_{\log}(f_j)
	\;\oplus\;
	\bigoplus_{i=1}^l \varOmega^1_{\sigma_i},\] then we are done. 
	
	If the question is not that straightforward, in order to study the direct sum above, consider $\omega\in\Omega^1_n(\mathbf{0})$ and, as before, write it as
	\[
	\omega=\omega_{ex}+\alpha_{ij}\sum_{i,j\in A}\omega_{ij}=\omega_{ex}+\hat{\omega}.
	\]
	For our specific subfamily of Hamiltonians, we can then apply Green's Theorem
	\[\oint_{\gamma_h}\omega=\oint_{\gamma_h}\hat{\omega}=\int_{{int(\gamma_h)}} \div (\hat{\omega})\,dxdy.\]
	With this, by studying properties of the divergence one can easily compute the vectorial subspace in $\Omega^1_n(\mathbf{0})$ of the $1$--forms that are algebraic relatively exact, or the vectorial subspace of the $1$--forms that share a symmetry of $H(x,y)$. Then, by Proposition \ref{prop:suficientes}, all the $1$--forms in the direct sum of these subspaces will have null integral. If furthermore one of the level curves of the Hamiltonian factorizes, the same applies for the correspondent logarithmic $1$--forms. 
	
	\section{Quadratic Hamiltonian differential equations}
	\label{sec:proofs-quad}

	In this section, we study the  Tangential Center Problem for quadratic Hamiltonian differential equations having a nondegenerate center at the origin. The goal is to prove Theorem~\ref{theo:mainqintro}, and we  do so in two propositions: Proposition~\ref{prop:genericq_tan} for the generic family \eqref{Hamilt-q}, and Proposition~\ref{prop:triangle_tan} for the triangular family \eqref{eq:HamT}. 
	
	As we have explained in Section~\ref{sec:prel},  the canonical forms 
	of cubic polynomials are given by  the polynomial
	$$
	H_N(x,y)=\frac{x^{2}+y^{2}}{2}
	+(x^{2}+y^{2})y
	+\frac{3x^{2}-y^{2}}{3}\,a y
	-\frac{x^{2}-3y^{2}}{3}\,b x,
	$$
	where $a,b\in\mathbb{R}$, for the generic case, and by the polynomial
	$$
	H_T(x,y)=\frac{x^2+y^2}{2}+x^2y-\frac{y^3}{3},
	$$
	for the Hamiltonian triangle.
	
	Note that the canonical form $H_N(x,y)$ is invariant under the transformation $(x,y,a,b)\rightarrow(-x,-y,-a,-b)$, so it is sufficient to consider the parameters $a$ and $b$ in the semiplane 
	\[
	\mathbb{H}(a,b):=\{(a,b)\in\mathbb{R}^2,b\geq0\}.
	\]
	
	\begin{prop}\label{prop:genericq_tan}
		Consider the quadratic Hamiltonian differential equation ${dH_N=0}$, associated with the polynomial $H_N(x,y)$, given in \eqref{Hamilt-q}. 
		
		\begin{enumerate}
			\item If  $b\neq 0$, then  the Tangential Center Problem for $dH_N=0$ has the  solution
			\[
			\TCL{2} = \varOmega^1_{e}.
			\] 
			
			\item If  $b=0$, then $H_N(x,y)$ has a orientation-reversing $\mathbb{Z}_2$-symmetry $\sigma(x,y)=(-x,y)$ and the Tangential Center Problem for $dH_N=0$ has the solution:
			\[
			\TCL{2} = \varOmega^1_{e}  \oplus \varOmega^1_{\sigma}=\varOmega^1_{e}  \oplus\langle\omega_1\rangle.
			\]
			where $\omega_1=xydx$.
		\end{enumerate}
	\end{prop}
	
	\begin{proof}
		Take $\omega\in\Omega^1_2(\mathbf{0})$. As shown in Section~\ref{sec:3}, there exist $\omega_{ex}\in\varOmega^1_e$ and $\hat\omega=(\alpha_{01}y+\alpha_{11}xy+\alpha_{02}y^2)dx$ such that
		\[
		\omega=\omega_{ex}+\hat{\omega}=\omega_{ex}+(\alpha_{01}y+\alpha_{11}xy+\alpha_{02}y^2)dx,
		\]
		so evidently
		\[I(h)=\oint_{\gamma_h}\omega=\oint_{\gamma_h}\hat{\omega}.\]
		For $\hat{\omega}$, we now apply our algorithm to express the Abelian integral $I(h)$ as a power series in $h$. In this proof, we provide a particularly detailed account of the computations involved in this process, as this is the first example illustrating the application of the algorithm.

		In polar coordinates, the Hamiltonian \eqref{Hamilt-q} writes as
		\[H(r,\theta)=\frac{r^{2}}{2}+
		\left( \sin(\theta)+a\frac{\sin(3\theta)}{3}-b\frac{\cos(3\theta)}{3} \right)r^3.\]
		Let $\gamma_h$ be the curve corresponding to the level $h^2/2$. Then $H(r,\theta)=h^2/2$ defines an analytic implicit function $r(h,\theta)$. Expanding in power series of $h$, 
		\[
		\begin{split}
			r(h,\theta) =& h -\frac{1}{3} h^2 (b \sin (3 \theta)-a \cos (3 \theta)+3 \sin (\theta)) \\
			&+\frac{5}{18} h^3 (b \sin (3 \theta)-a\cos (3\theta)+3 \sin(\theta))^2\\
			&-\frac{8}{27} h^4 (b \sin (3 \theta)-a \cos (3 \theta)+3 \sin(\theta))^3\\
			&+\frac{77}{216} h^5 (b \sin (3 \theta)-a \cos (3 \theta)+3 \sin(\theta))^4 +
			\mathcal{O}(h^6).
		\end{split}
		\]
		Since
		\[
		\begin{split}
			I(h) &=\alpha_{01} \int_0^{2\pi}
			r(h,\theta)\sin(\theta)\,d (r(h,\theta)\cos(\theta)) 
			\\ 
			& +   \alpha_{11} \oint_{\gamma_h} 
			r^2(h,\theta)\cos(\theta)\sin(\theta) \,d (r(h,\theta)\cos(\theta))
			\\ 
			& + \alpha_{02} \oint_{\gamma_h} 
			r^2(h,\theta)\sin^2(\theta)\,d (r(h,\theta)\cos(\theta)),
		\end{split}
		\]
		we can expand $I(h)=\displaystyle\sum_{\mu=1}^\infty K_\mu h^{\mu}$, for example, up to power six, and we get 
		\[
		\begin{aligned}
			I(h)
			=&-\pi\alpha_{01}h^2
			+\pi\left(
			2\alpha_{02}
			+\alpha_{01}c_1
			\right)h^4+\pi\left(
			\alpha_{01}c_2
			+\frac{10}{3}\alpha_{11}b
			+\alpha_{02}c_3\right)h^6
			+O(h^7),
		\end{aligned}
		\]
		for certain coefficients $c_1,c_2,c_3$.
		Clearly, $K_2=K_4=K_6=0$ if and only if $\alpha_{01}=\alpha_{02}=\alpha_{11} b=0$. If $b\neq 0$, then $\alpha_{01}=\alpha_{02}=\alpha_{11}=0$, that is, $\omega=\omega_{ex}$ and consequently the only solutions of the tangential problem are the exact $1$--forms. This proves statement (1).
		
		If  $b=0$, then we  get $K_\mu=0$ for any $\mu$ we compute, and thus the candidates for solutions of the tangential problem are $\omega=\omega_{ex}+\alpha_{11}xydx$. We need to prove that the Abelian integral of such $1$--forms is null. 
		
		Recall that when $b=0$, $H(-x,y)=H(x,y)$, and that $xydx$ remains invariant under the change $x\rightarrow-x$. Thus, we are done by applying Proposition \ref{prop:suficientes} taking as $\sigma$ the reflection with respect to the $y$ axis. 
	\end{proof}
	
	We prove now the equivalent result but for the triangle family.
	\begin{prop}\label{prop:triangle_tan}
		Consider the quadratic Hamiltonian differential equation ${dH_T=0}$, associated with the polynomial $H_T(x,y)$, given in  \eqref{eq:HamT}. Then, $H_T(x,y)$ has a orientation-reversing $\mathbb{Z}_2$-symmetry $\sigma(x,y)=(-x,y)$, and an orientation-preserving $\mathbb{Z}_3$-symmetry, $\tau(x,y)=\left(
		-\frac12 x-\frac{\sqrt3}{2}y,\;
		\frac{\sqrt3}{2}x-\frac12 y
		\right)$; and the Tangential Center Problem for $dH_T=0$ has the solution:
		\[
		\TCL{2}=\varOmega^1_{e}  \oplus \varOmega^1_{\sigma} \oplus \varOmega^1_{\tau}=\varOmega^{1}_{e}
		\oplus
		\langle \omega_1\rangle
		\oplus \langle \omega_2\rangle,
		\]
		where $\omega_1=xydx$, $\omega_2=y^2dx$.
		
	\end{prop}
	
	\begin{proof}
		
		Take $\omega\in\varOmega^1_2(\mathbf{0})$. It can be written, for some $\omega_{ex}\in\varOmega^1_e$, as
		\[
		\omega=\omega_{ex}+\hat{\omega}=\omega_{ex}+(\alpha_{01}y+\alpha_{11}xy+\alpha_{02}y^2)dx.
		\]
		Therefore
		\[I(h)=\oint_{\gamma_h}\omega=\oint_{\gamma_h}\hat{\omega}.\]
		For $\hat{\omega}$, we now apply our algorithm to express the Abelian integral $I(h)$ as a power series in $h$, that is,
		$$
		I(h)=\displaystyle\sum_{\mu=1}^\infty K_{\mu}h^{\mu}.
		$$
		
		In particular, we get $K_1=0$ always, and $K_2=0$ if and only if $\alpha_{01}=0$. If we impose this, then $K_\mu=0$ for any $\mu$ we compute. Thus, the algorithm is suggesting that any $1$--form $\omega=\omega_{ex}+(\alpha_{11}xy+\alpha_{02}y^2)dx$ is a candidate for a solution of the tangential problem. We now need to prove that the Abelian integral is indeed zero for perturbations like this. 
		
		Note that the Hamiltonian \eqref{eq:HamT} possesses two symmetries: reflection with respect to the $y$-axis and rotation by $2\pi/3$ radians. It is a straightforward computation to check that we can apply Proposition \ref{prop:suficientes} to $\omega_1 = xy dx$ for the reflection, so $\omega_1 \in \varOmega^1_{e}  \oplus \varOmega^1_{\sigma}$. Moreover,  for $\omega_2=y^2dx$ we have
		\[
		\omega_2 + \tau^*\omega_2 + (\tau^2)^*\omega_2
		=
		\frac34(y^2-x^2)\,dx+\frac32xy\,dy
		=
		d\!\left(\frac34xy^2-\frac14x^3\right),
		\]
		and, therefore, $\omega_2 \in \varOmega^1_{e}  \oplus \varOmega^1_{\tau}$.
		Consequently $I(h)\equiv 0$ for those perturbations and we are done.  
	\end{proof}
	
	\begin{rema}\label{rema:centers2}
		
		The classification obtained above admits a natural interpretation in terms of the classical center varieties for quadratic equations. Recalling Iliev \cite{Iliev} normal forms for quadratic families with a center in the origin, based on \.Zo{\l}adek's classification \cite{Zoladek}, one has that, besides the Hamiltonian families considered in this work, the remaining center components are given by: 
		\begin{enumerate} 
			
			\item \textbf{Reversible family} \[ \begin{cases} \dot x = y+(a+b+2)x^2+(2-a-b)y^2,\\[1ex] \dot y = -x+2(a-b-1)xy. \end{cases} \] 
			
			\item \textbf{Codimension--four family} \[ \begin{cases} \dot x = y+\left(4+\frac{4}{a}+a\right)x^2 +\left(\frac{4}{a}-a-4\right)y^2,\\[1ex] \dot y = -x+2(4-a)xy, \end{cases} \qquad a\neq 0. \]
			
			\item \textbf{Generalized Lotka--Volterra family} \[ \begin{cases} \dot x = y+(1+a)x^2-2bxy+(a-1)y^2,\\[1ex] \dot y = -x+2(1-a)xy-b(x^2-y^2). \end{cases} \] \end{enumerate} 
		
		A direct comparison of coefficients shows that the generic Hamiltonian family intersects the reversible family precisely when \[ b=0, \] while it has no intersection with either the codimension--four or the generalized Lotka--Volterra families. Furthermore, the Hamiltonian triangle belongs simultaneously to the reversible and generalized Lotka--Volterra families.
		
		Thus, the exceptional cases detected by the Tangential Center Problem occur precisely at the loci where different center families meet. This provides a geometric interpretation of the classification obtained here. Indeed, the condition $b=0$ corresponds to the intersection between the generic Hamiltonian and reversible center families, whereas the Hamiltonian triangle is a distinguished case belonging simultaneously to the reversible and generalized Lotka--Volterra families. Consequently, the decomposition \[ b\neq0,\qquad b=0,\qquad \text{Hamiltonian triangle}, \] obtained for the tangential center locus mirrors the stratification induced by the intersections of the irreducible components of the classical center variety.
		
		It should be noted, however, that the mechanisms responsible for tangential centers do not necessarily coincide with those responsible for ordinary centers. While the cases $b=0$ and the Hamiltonian triangle naturally arise from intersections with non-Hamiltonian center families, the vanishing of the corresponding Abelian integrals is explained by additional geometric properties of the Hamiltonian equations under consideration. In particular, the Hamiltonian triangle possesses a rotational symmetry of order three (a $120^\circ$ rotation), which forces the vanishing of the Abelian integrals appearing in the Tangential Center Problem. This symmetry is not, by itself, a center condition in the classical sense. 
		
		Nevertheless, the fact that the classification obtained from the Tangential Center Problem reproduces exactly the pattern of intersections among \.Zo{\l}adek's center families suggests a close connection between the geometry of the center variety and the structure of the tangential center locus.
	\end{rema}

	\section{Cubic Hamiltonian differential equations}
	\label{sec:proofs-cubic}

	In this section, we  apply the method explained in Section~\ref{sec:3}  to study the Tangential Center Problem for a family of cubic Hamiltonian equations. 
	As we mention in the introduction,  we  study the family  \eqref{eq:Hamce0}, that we recall here:
	\begin{equation*}\label{eq:Hamce}
		H(x,y)=\frac{x^2+y^2}{2}+x^2y-\frac{y^3}{3}+ax^4+bx^2y^2+cy^4.
	\end{equation*}
	This is a generalization of the Hamiltonian triangle~\eqref{eq:HamT} by adding homogeneous quartic terms to the Hamiltonian with only even exponents.
	
	We also recall that by the results of Ilyashenko \cite{Ilyashenko} if the equation has $n^2$ different critical points all of them at different energy levels, then the only solution to the tangential center problem are exact $1$-forms. Therefore, the problem can be restricted to the families where that condition is not satisfied. Let $\sigma(x,y)=(-x,y)$. As $H(x,y)=H(\sigma(x,y))$ for any $(x,y)$, then the entire family~\eqref{eq:Hamce0} is symmetric with respect to the $y$ axis, and, in particular, there are always a pair of critical points with the same critical values.

	As in Section~\ref{sec:proofs-quad}, we  restate the main result, Theorem~\ref{theo:cubic} and prove it. 
	\begin{prop}
		Consider the cubic Hamiltonian differential equation ${dH=0}$, associated with the polynomial $H(x,y)$, given in \eqref{eq:Hamce0}, which always has an  orientation-reversing $\mathbb{Z}_2$-symmetry $\sigma(x,y)=(-x,y)$. Then,
		\begin{enumerate}
			\item If $c=a$, $b=2a$, $H(x,y)$ has a orientation-preserving $\mathbb{Z}_3$-symmetry, $\tau(x,y)=\left(
			-\frac12 x-\frac{\sqrt3}{2}y,\;
			\frac{\sqrt3}{2}x-\frac12 y
			\right)$; and the Tangential Center Problem for $dH=0$ has the solution:
			\[ 
			\TCL{3}=\varOmega^1_{e}  \oplus \varOmega^1_{\sigma} \oplus \varOmega^1_{\tau}=\varOmega^1_{e}  \oplus \varOmega^1_{\sigma}
			\oplus
			\langle \omega_1\rangle
			\oplus \langle \omega_2\rangle ,
			\]
			where $\omega_1=y^2dx$ and $\omega_2=y(y^2-3x^2)dx$.
			\item If $c\neq a$, $b=c=0$, then the Tangential Center Problem for $dH=0$ has the solution:
			\[
			\TCL{3}=\varOmega^1_{re}  \oplus \varOmega^1_{\sigma}=\varOmega^1_{e}  \oplus \varOmega^1_{\sigma}
			\oplus
			\langle \omega_3\rangle,
			\]
			where $\omega_3=\left(y^2-\frac{2}{3}y(y^2-3x^2)\right)dx$.
			\item If $c\neq a$ and 
			\[a+b+c=18 \, b^{2} c + 72 \, b c^{2} + 72 \, c^{3} - b^{2} - 2 \, b c + 8 \, c^{2}=0,\]
			then $H(x,y) - k(a,b,c) = f_1 f_2 f_3$, with $k$ a certain constant depending on $a,b,c$ and $f_1,f_2$ linear and $f_3$ quadratic, and the Tangential Center Problem for $dH=0$ has the solution:
			\[
			\TCL{3}=\varOmega^1_{e}  \oplus \varOmega^1_{\sigma} \oplus \varOmega^1_{\log}(f_1)\oplus \varOmega^1_{\log}(f_2)=\varOmega^1_{e}  \oplus \varOmega^1_{\sigma}
			\oplus
			\langle \omega_4\rangle,
			\]
			where $\omega_4=\left(y^2+\frac{2(3a+b-5c)}{9(a-c)}y(3x^2-y^2)\right)dx$.
			\item In any other case,  the Tangential Center Problem for $dH=0$ has the solution $\TCL{3}=\varOmega^1_{e}  \oplus \varOmega^1_{\sigma}$.
			
		\end{enumerate}
		
	\end{prop}
	\begin{proof}
		For this family of Hamiltonians, we know that every $\omega\in\varOmega^1_{e}  \oplus \varOmega^1_{\sigma}$ is a solution of the tangential problem. Therefore, it is convenient to use a refined version of \eqref{basispetrov}, in the following way:  consider $\omega = P(x,y)\,dx + Q(x,y)\,dy\in\Omega^1_{3}(\mathbf{0})$, where
		\begin{equation*}\label{cub1form}
			\begin{split}
				P(x,y) &= a_1 x + a_2 y + a_3 x^2 + a_4 y^2 + a_5 xy + a_6x^3 + a_7y^3 + a_8x^2y + a_9xy^2, \\
				Q(x,y) &= b_1 x + b_2 y + b_3 x^2 + b_4 y^2 + b_5 xy + b_6x^3 + b_7y^3 + b_8x^2y + b_9xy^2.
			\end{split}
		\end{equation*}
		
		Since
		\[
		d\omega=(Q_x-P_y)\,dx\wedge dy,
		\]
		if we want $\omega\in \Omega^1_e\oplus \Omega^1_\sigma$, then $(Q_x-P_y)(x,y)=-(Q_x-P_y)(-x,y)$. 
		If we impose these conditions, we get
		\[a_8=3b_6,\qquad 3a_7=b_9,\qquad 2a_4=b_5,\qquad a_2=b_1.\]
		
		Define $\omega_s=P_s\,dx+Q\,dy$, where
		\[
		P_s
		=
		P-\Bigl(a_2-b_1\Bigr)y
		-\Bigl(a_4-\frac{b_5}{2}\Bigr)y^2
		-\Bigl(a_8-3b_6\Bigr)x^2y
		-\Bigl(a_7-\frac{b_9}{3}\Bigr)y^3.
		\]
		Then $\omega_s\in \Omega^1_e\oplus \Omega^1_\sigma$, and
		\[
		\omega-\omega_s=
		\left(
		(a_2-b_1)y+\left(a_4-\frac{b_5}{2}\right)y^2+(a_8-3b_6)x^2y+\left(a_7-\frac{b_9}{3}\right)y^3
		\right)dx.
		\]
		That is, renaming the coefficients, we can always write any $1$--form $\omega\in\Omega^1_3(\mathbf{0})$ as:
		\begin{equation*}\label{cub1f}
			\omega=\omega_s+\hat{\omega}=\omega_{s}+(Ay+By^2+Cx^2y+Dy^3)dx,
		\end{equation*}
		so
		\[
		I(h)=\oint_{\gamma_h}\omega=\oint_{\gamma_h}\hat{\omega}.
		\]
		Now, we apply our algorithm to $\hat{\omega}$ to express the Abelian integral $I(h)=\displaystyle\sum_{\mu=1}^\infty K_\mu h^\mu$ as a power series in $h$.
		
		In particular, we get $K_\mu=0$ for any $\mu$ odd we compute, and that $K_2=0$ if and only if $A=0$. If we impose this, then $K_4=0$ if and only if $C=-3D$.
		
		As the solutions of the tangential problem have $K_\mu=0$ for all $\mu$, we can impose these two conditions, and our perturbation now writes as:
		\[
		\omega=\omega_s+\big(By^2+D(y^3-3x^2y)\big)\,dx
		\]
		We proceed now to compute $K_6$, imposing the two previous assumptions, and get
		\[
		K_6=k_6\left(24D(a-c)-\frac{16}{3}(3a+b-5c)B\right).
		\]
		where $k_6\neq0$ is a constant. Following the algorithm stated in Section \ref{sec:3}, we divide the study into two cases.
		
		{\bf Case 1:} $\mathbf{c=a}.$ We have
		\[
		K_6=k_6\left(\frac{32}{3}(2a-b)B\right).
		\]
		And here we have two subcases.
		
		{\bf Case 1.1:} $\mathbf{b\neq 2a}.$ In this case, $K_6=0$ if and only if $B=0$. Imposing this, we compute $K_8$, and get that $K_8=0$ if and only if $D=0$. So, in this case, the only perturbations that solve the tangential problem are $\omega=\omega_s$.
		
		{\bf Case 1.2:} $\mathbf{b=2a}.$ Here, $K_6=0$, and the algorithm already yields $K_\mu=0$ for all $\mu\leq 16$. These perturbations are thus a good candidate for a solution of the tangential problem. Our Hamiltonian writes in this case as
		\[
		H_1(x,y)=\frac{x^2+y^2}{2}+x^2y-\frac{y^3}{3}+a(x^2+y^2)^2.
		\]
		In this case, the Abelian integral is:
		\[
		I(h)=B\int_{\gamma_h}y^2dx+D\int_{\gamma_h}y(y^2-3x^2)dx=B\,I_1(h)+D\,I_2(h).
		\]
		where $\gamma_h$ are level curves of $H_1$.
		
		Now, $\int_{\gamma_h}\omega_1=\int_{\gamma_h}y^2dx=0$ following the same reasoning as in Proposition~\ref{prop:triangle_tan} for the same perturbation, since $H_1$ possesses the same symmetries as the cubic Hamiltonian triangle~\eqref{eq:HamT}.
		
		On the other hand, writing $\omega_2=y(y^2-3x^2)dx,$ we have
		\[\omega_2 + \tau^*\omega_2 + (\tau^2)^*\omega_2
		=0,\]
		and so, applying Proposition \ref{prop:suficientes} to $H_1$ and $\omega_2,$ and using the symmetry with respect to $\tau$, we prove that $\int_{\gamma_h}\omega_2=\int_{\gamma_h}y(y^2-3x^2)dx=0$ and, in conclussion, that $I(h)=0$.

	{\bf Case 2:} $\mathbf{c\neq a}.$ In this case, $K_6=0$ if 
	and only if $D=\big(2(3a+b-5c)/9(a-c)\big)B$. If we impose this, we have reduced the Abelian integral to
	\[I(h)=B\int_{\gamma_h}\left(y^2+\frac{2(3a+b-5c)}{9(a-c)}y(3x^2-y^2)\right)dx.\]
	
	We can now compute $K_8$, and we get:
	\[
	K_8=k_8\frac{B}{c-a}T(a,b,c),
	\]
	where $k_8$ is a non-zero constant and 
	\[T(a,b,c)=3 a^{2}( b +10 c) +3 a b(b-2c)+c^{2}(3b -72 a+42 c+10)+b(3a +b-7 c)-3 b^{2} c -6 c a.\]
	
	From this, we do not get a simple relationship between $a,b,c$. As we have, thanks to the algorithm, removed various parameters from the perturbation, it is computationally feasible to compute $K_\mu$, for $\mu=8,\ldots,16,$ impose all of them to be zero, and, using the algebraic manipulator Singular, compute the associated minimal prime ideals. We get
	\begin{enumerate}[$i)$]
		\item $B=0$.
		\item $b=c=0$.
		\item $a+b+c=18 \, b^{2} c + 72 \, b c^{2} + 72 \, c^{3} - b^{2} - 2 \, b c + 8 \, c^{2}=0$.
	\end{enumerate}
	The first one is $B=0$, and in this case we are left only with exact and reversible perturbations as a solution to the problem. 
	
	The second possibility is 
	$b=c=0$, and $a\neq 0$ in order to have cubic terms at the Hamiltonian. In this case, the Hamiltonian is 
	\[
	H_2(x,y)=\frac{x^2+y^2}{2}+x^2y-\frac{y^3}{3}+ax^4,
	\]
	and the Abelian integral simplifies to
	\[
	I(h)=B\int_{\gamma_h} \left(y^2-\frac{2}{3}y(y^2-3x^2)\right)dx.
	\]
	Note that 
	\[
	\omega_3=(y^2-\frac{2}{3}y(y^2-3x^2))dx= (2H_2(x,y)-(x^2+2ax^4))dx.
	\]
	Then,
	\[
	\begin{split}
		\int_{\gamma_h}\omega_3=\int_{\gamma_h} \left(y^2-\frac{2}{3}y(y^2-3x^2)\right)dx&=\int_{\gamma_h}(2H_2(x,y)-(x^2+2ax^4))dx \\
		&=\int_{\gamma_h}(2h-(x^2+2ax^4))dx.
	\end{split}
	\]
	And this last integral is zero by Green's Theorem. Thus, $I(h)=0$ for this perturbation.
	
	In the last case, that is, 
	\[a+b+c=18 \, b^{2} c + 72 \, b c^{2} + 72 \, c^{3} - b^{2} - 2 \, b c + 8 \, c^{2}=0,\]
	we  use $b=-a-c$. So, the second condition becomes
	$$
	18a^2c-36ac^2+18c^3-a^2+9c^2=0,
	$$
	and the Abelian integral is reduced to
	\[
	I(h)=B\int_{\gamma_h}\left(y^2+\frac{4(a-3c)}{9(a-c)}y(3x^2-y^2)\right)dx.
	\]
	
	The second condition, in the parameters $a$ and $c$,
	represents a rational curve, whose parameterization can be given by 
	\[
	s \longmapsto \left(\frac{s(s^2-9)}{18(s-1)^2}, \frac{s^2-9}{18(s-1)^2}\right),\quad s\neq 1,-3.
	\]
	Therefore, in this case we have the 1-parameter family of Hamiltonians
	$$
	H_s(x,y)=\frac{x^2+y^2}{2}+x^2y-\frac{y^3}{3}+\frac{(s^2-9)}{18(s-1)^2}\Big(s x^4-(s+1)x^2y^2+ y^4\Big),
	$$
	and the Abelian integral becomes
	\[
	I(h)=B\int_{\gamma_h}\left(y^2
	+\frac{4(s-3)}{9(s-1)}y(3x^2-y^2)\right)dx.
	\]
	Set $h_0:=9(s-1)^2/(s+3)^3$.
	We have the factorization
	$$
	\widetilde{H}_s(x,y)=H_s(x,y)-h_0=f_1 f_{21} f_{22},
	$$
	where
	$$
	f_1=\frac{\left( {s}^{2}-9 \right) {x}^{2}}{18 \left( s-1 \right)^{2}}-\frac{ \left( {s}^{2}-9 \right) {y}^{2}}{18 \left( s-1 \right)^{2}}+\frac{2y}{3(s-1)}+ \frac{1}{s+3},
	$$
	$$
	f_{21}=\sqrt{s}x-y+\frac{3(s-1)}{s+3},
	\quad
	f_{22}=\sqrt{s}x+y-\frac{3(s-1)}{s+3}.
	$$
	Note that when $s<0$, the polynomials $f_{21}$ and $f_{22}$ are complex. We choose the branch
	$\sqrt{s}=i\sqrt{|s|}$ when $s<0$; the choice is irrelevant, since changing the sign of
	$\sqrt{s}$ exchanges $f_{21}$ with $-f_{22}$ and $f_{22}$ with $-f_{21}$, and changes the
	sign of both the prefactor and $\eta_1-\eta_2$ in the identity below.
	
	Assume $s\neq 0$. Defining the $1$--forms
	$$
	\eta_1=\widetilde{H}_s(x,y)\frac{df_{21}}{f_{21}},\quad
	\eta_2=\widetilde{H}_s(x,y)\frac{df_{22}}{f_{22}}
	$$
	(these are polynomial $1$--forms: $\eta_1=f_1f_{22}\,df_{21}$ and
	$\eta_2=f_1f_{21}\,df_{22}$, the poles being only apparent) and considering the polynomial
	$$
	R(x,y)=\frac{(s-3){x}^{3}y}{3(s-1)}+\frac{( 
		s-3) {x}^{3}}{3(s+3)}-\frac{( s-3 ) x{y}^{3}}{
		9(s-1)}+\frac{2x{y}^{2}}{s+3}+\frac{6( s-1) xy}{
		\left( s+3 \right) ^{2}}+\frac{18(s-1) ^{2}x}{
		\left( s+3 \right) ^{3}},
	$$
	one checks that 
	$$
	\left(y^2
	+\frac{4(s-3)}{9(s-1)}y\big(3x^2-y^2\big)\right)dx=\frac{3(s-1)}{\sqrt {s} \left( s+3 \right)}(\eta_1-\eta_2)+dR(x,y).
	$$
	Hence, since $dR(x,y)$ is an exact polynomial $1$--form,
	$$
	I(h)=B\int_{\gamma_h}\left(y^2
	+\frac{4(s-3)}{9(s-1)}y(3x^2-y^2)\right)dx=
	B\int_{\gamma_h}\frac{3(s-1)}{\sqrt {s} \left( s+3 \right)}(\eta_1-\eta_2).
	$$
	Now, restricting to $\gamma_h$, where $\widetilde{H}_s=h-h_0$, we obtain
	$$
	I(h)=\frac{3B(s-1)}{\sqrt {s} \left( s+3 \right)}(h-h_0)\left(\int_{\gamma_h}\frac{df_{21}}{f_{21}}-\int_{\gamma_h}\frac{df_{22}}{f_{22}}\right), 
	$$
	which vanishes identically because the cycle $\gamma_h$ does not turn around neither the (real or complex) straight line $\{f_{21}=0\}$ nor $\{f_{22}=0\}$.

	Assume $s=0$, so that $f_{21}=-(y+1)$, $f_{22}=-f_{21}$ and $h_0=1/3$: the two lines have
	coalesced into the double line $\{y=-1\}$, one has $\eta_1=\eta_2$, and the previous
	argument degenerates. Define
	\begin{equation*}\label{eq:omega0}
		\omega_0:=d\Bigl(\frac{x}{f_{21}}\Bigr)
		=\frac{dx}{f_{21}}-x\,\frac{df_{21}}{f_{21}^{2}},
		\qquad
		\eta:=\widetilde{H}_0\,\omega_0
		=\widetilde{H}_0\,\frac{dx}{f_{21}}-x\,\widetilde{H}_0\,\frac{df_{21}}{f_{21}^{2}},
	\end{equation*}
	and
	\[
	R_0(x,y)=x^3y-\frac{x^3}{3}-\frac{xy^3}{3}+\frac{2xy^2}{3}-\frac{2xy}{3}+\frac{2x}{3}.
	\]
	Note that, again, $\eta$ is a polynomial $1$-form: the double pole of $\omega_0$ along
	$\{f_{21}=0\}$ is only apparent, because
	$\widetilde{H}_0=f_1f_{21}f_{22}=-f_1f_{21}^{2}$ vanishes doubly along the line. Since
	\[
	\Bigl(y^2
	+\frac{4}{3}y\big(3x^2-y^2\big)\Bigr)dx=-2\,\widetilde H_0\,\frac{dx}{f_{21}}
	+2x\,\widetilde H_0\,\frac{df_{21}}{f_{21}^{2}}
	+dR_0
	=-2\,\eta+dR_0
	\]
	then
	\[
	I(h)=B\int_{\gamma_h}\bigl(-2\,\eta+dR_0\bigr)=-2B\Bigl(h-\frac13\Bigr)\int_{\gamma_h} d\Bigl(\frac{x}{f_{21}}\Bigr),
	\]
	as $dR_0$ is exact and on $\gamma_h$ one has $\widetilde{H}_0=h-h_0$. The previous integral is identically zero
	because the oval $\gamma_h$ does not intersect $\{f_{21}=0\}=\{y=-1\}$ (in fact, for these values of the parameters, the line $y=-1$ is full of critical points and hence the period annulus of the center does not meet this line). Consequently, 
	$\varphi:=x/f_{21}$ is smooth near the oval $\gamma_h$ and 
	$d\varphi$ is exact there.
\end{proof}

\begin{rema}\label{rema:centers3}
	
	As in the quadratic case, the classification obtained for the Tangential Center Problem also provides information about the Infinitesimal Center Problem. 
	
	The family~\eqref{eq:Hamce0} lies in the intersection of the Hamiltonian and reversible planar equations. Therefore, the generic case (4) corresponds to this intersection, and the solution of the Tangential Center Problem is the direct sum of the tangent spaces associated with the Hamiltonian and reversible subfamilies. 
	
	In case (3), namely, \[ a+b+c=18 b^{2} c + 72 b c^{2} + 72 c^{3} - b^{2} - 2 b c + 8 c^{2}=0, \] we obtain the intersection with the family of centers having two invariant lines and one invariant quadric, with transversal intersection. 
	
	In case (2), where $c\neq a$ and $b=c=0$, we have that \[ dH_2+\ep C\left(y^2-\frac{2}{3}y(y^2-3x^2)\right)=0 \] is Darboux integrable, and its first integral is \begin{multline*} \Big(24a{\ep}^{3}{C}^{3}{x}^{3} +12{\ep}^{4}{C}^{4}{x}^{2}y -36a{\ep}^{2}{C}^{2}{x}^{2} +6\ep C \left({\ep}^{2} {C}^{2}+6a \right) x\\ -4{\ep}^{4}{C}^{4}{y}^{3} +6{\ep}^{4}{C}^{4}{y}^{2} -3{\ep}^{2}{C}^{2}-18a \Big) {\rm e}^{2\ep C x}, \end{multline*} so this subcase yields an intersection with another family of centers.
	
	Finally, in case (1), where $a=c$ and $b=2a$, we have not found a simple family of centers intersecting our family. We therefore leave as an open problem the question of whether such a family exists.
	
\end{rema}

\section*{Acknowledgments}

The authors are partially supported by grant number PID2023-151974NB-I00 funded by MICIU/AEI/10.13039/501100011033/FEDER, UE. The second and third authors have been partially supported by the European Union, the European Regional Development Fund, and the Regional Government of Extremadura, in collaboration with the Managing Authority and the Ministerio de Hacienda, through Project GR24042.

	\end{document}